\documentclass[11pt]{amsart}
\usepackage{amssymb, amsmath}
\usepackage{amsthm}
\usepackage{amscd}
\usepackage{graphicx}
\newtheorem{theorem}{Theorem}
\newtheorem{lemma}[theorem]{Lemma}

\newtheorem{proposition}[theorem]{Proposition}
\newtheorem{conjecture}[theorem]{Conjecture}
\theoremstyle{definition}

\newtheorem*{remark}{Remark}
\newtheorem*{acknowledgement}{Acknowledgement}
\newtheorem{definition}[theorem]{Definition}
\title[Some examples of noncompact nonpositively curved manifolds]
{Examples of noncompact nonpositively curved manifolds}
\author{Grigori Avramidi, T. T$\hat{\mathrm{a}}$m Nguy$\tilde{\hat{\mathrm{e}}}$n-Phan}
\address{Mathematische Institut\\Universit\"at M\"unster\\ Germany}
\email{avramidi@uni-muenster.de}
\address{Max Planck Institute for Mathematics\\
Bonn\\
Germany}
\email{tam@mpim-bonn.mpg.de}

\def\ra{\rightarrow}

\def\beqa{\begin{eqnarray}}
\def\eeqa{\end{eqnarray}}
\def\beqa{\begin{eqnarray}}
\def\eeqa{\end{eqnarray}}

\DeclareMathOperator{\id}{Id}

\DeclareMathOperator{\funnel}{funnel}

\def\R{\mathbb{R}}

\def\Q{\mathbb{Q}}

\input xy
\xyoption{all}

\begin{document}

\begin{abstract}
We give a simple construction of new, complete, finite volume manifolds $M$ of bounded, nonpositive curvature. These manifolds have ends that look like a mixture of locally symmetric ends of different ranks and their fundamental groups are not duality groups. 

\end{abstract}
\maketitle

\section{Introduction}
The goal of this paper is to give a very simple construction of complete, finite volume, tame\footnote{``Tame" means that the manifold is homeomorphic to the interior of a compact manifold-with-boundary.} $n$-manifolds $M$ of bounded, nonpositive curvature. The manifolds obtained have interesting properties. For instance, the large scale geometry of their ends is a mixture of different types and their fundamental groups are not duality groups, in contrast with the typical examples of nonpositively curved manifolds such as locally symmetric spaces of noncompact type. If $M$ is a locally symmetric manifold of noncompact type, then from  a large-scale point of view $M$ looks like a union of flat $r$-dimensional sectors, where $r$ is the $\Q$-rank of $M$. So for (arithmetic\footnote{All irreducible higher rank locally symmetric spaces are arithmetic by Margulis' arithmeticity theorem (\cite{zimmer}).}) locally symmetric spaces, their large-scale geometry is determined by the rational structures of the spaces. 
Moreover, the fundamental group of $M$ is a duality group, or in other words, the lift of the end of $M$ to the universal cover $\widetilde{M}$ has homology concentrated in one dimension. This is a consequence of the fact that it is homotopy equivalent to the rational Tits building (of $M$), which is homotopy equivalent to a wedge of spheres of a single dimension. 
\newline

In \cite{anphalfcollapse}, we tried to capture the topology of the ends of general nonpositively curved, not necessarily locally symmetric, manifolds $M$ from the geometry of $M$ and  $\widetilde{M}$, showing that many properties of locally symmetric manifolds that could be seen only by doing arithmetic before can actually be seen as purely nonpositive curvature phenomena. For example, we obtained that the lift of the end of $M$ in $\widetilde{M}$ has homology only in dimension less than $n/2$. In other words, it is no arithmetic coincidence that the rational Tits building of a locally symmetric space has dimension less than half the dimension of the space. However, one cannot take this analogy too far and base all aspects of about nonpositively curved manifolds on delicacies of locally symmetric spaces because there are still arithmetic things that are due to arithmetics, such as the rational Tits building being a building, and this is one of the main points of the examples in this paper. 
\newline

Below, $M$ is tame, so it is homeomorphic to the interior of a compact manifold-with-boundary, $(\overline M,\partial\overline M)$ and its universal cover is a (non-copmact) manifold-with-boundary $(\widetilde{\overline M},\partial\widetilde{\overline M})$. We will abuse notation slightly and denote these manifolds-with-boundary by $(M,\partial M)$ and $(\widetilde M,\partial\widetilde M)$, respectively. Note that $\partial\widetilde M\ra\partial M$ is regular cover with covering group $\pi_1M$, so we call it the $\pi_1M$-cover of $\partial M$.     

\begin{theorem}
\label{maintheorem}
For any $0\leq i\leq j<\lfloor n/2\rfloor$ there is a tame, complete, finite volume, Riemannian $n$-manifold $M$ of bounded nonpositive curvature with the property that $\overline H_{k}(\partial\widetilde M)\not=0$ if and only if $i\leq k\leq j$.
\end{theorem}
In fact, in our examples, $\partial\widetilde M$ is homotopy equivalent to a union of wedges of spheres of dimensions ranging from $i$ to $j$.
\begin{remark}
One can show (see section 10.4 of \cite{anphalfcollapse}) that $$\overline H_*(\partial\widetilde M)\cong H^{n-1-*}(B\pi_1M;\mathbb Z\pi_1M),$$ so as an algebraic corollary, $\pi_1M$ is not a duality group if $j>i$.
\end{remark} 
\bigskip


\noindent
\textbf{The construction} is done inductively and the main idea is to assemble nonpositively curved spaces like products of hyperbolic manifolds with cusps via codimension $2$ surgery along totally geodesic submanifolds. As usual, one needs to smooth out the metric around the places where surgery is done, but in this case, one can do the smoothing of the metric blind folded and with both hands tied behind one's back. 
This is because at each step, we choose suitable manifolds $M_1^k$ and $M_2^k$, each of which has an open set that is isometric to $\mathbb{T}^{k-2}\times \mathbb{D}^2$, where $\mathbb{T}^{k-2}$ is the flat square torus. Then we remove $\mathbb{T}^{k-2}\times\mathbb{D}_{\varepsilon}^2$
 from each $M_i$ and glue the resulting spaces together along the boundary preserving the product structure on $\mathbb{T}^{k-2}\times \partial\mathbb{D}^2_\varepsilon$ to obtain a new manifold $M$ whose metric is singular on $\mathbb{T}^{k-2}\times \partial\mathbb{D}^2_\varepsilon$. Since the gluing is an isometry on the first factor, the singularity of the metric lies in the second factor, which is the double of $(\mathbb{D}^2-\mathbb{D}^2_\varepsilon)$ along $\partial\mathbb{D}^2_\varepsilon$. To smooth out this singularity, replace this double by a ``funnel" that is the surface of revolution generated by the curve $\alpha$ in Figure \ref{graph}, which clearly has nonpositive Gaussian curvature. Thus, we obtain a bounded nonpositively curved manifold $M$ whose ends correspond to those of $M_1$ and $M_2$ and therefore have finite volume.
\newline

\begin{figure}
\label{graph}
\centering
\includegraphics[scale=0.07, angle = 90]{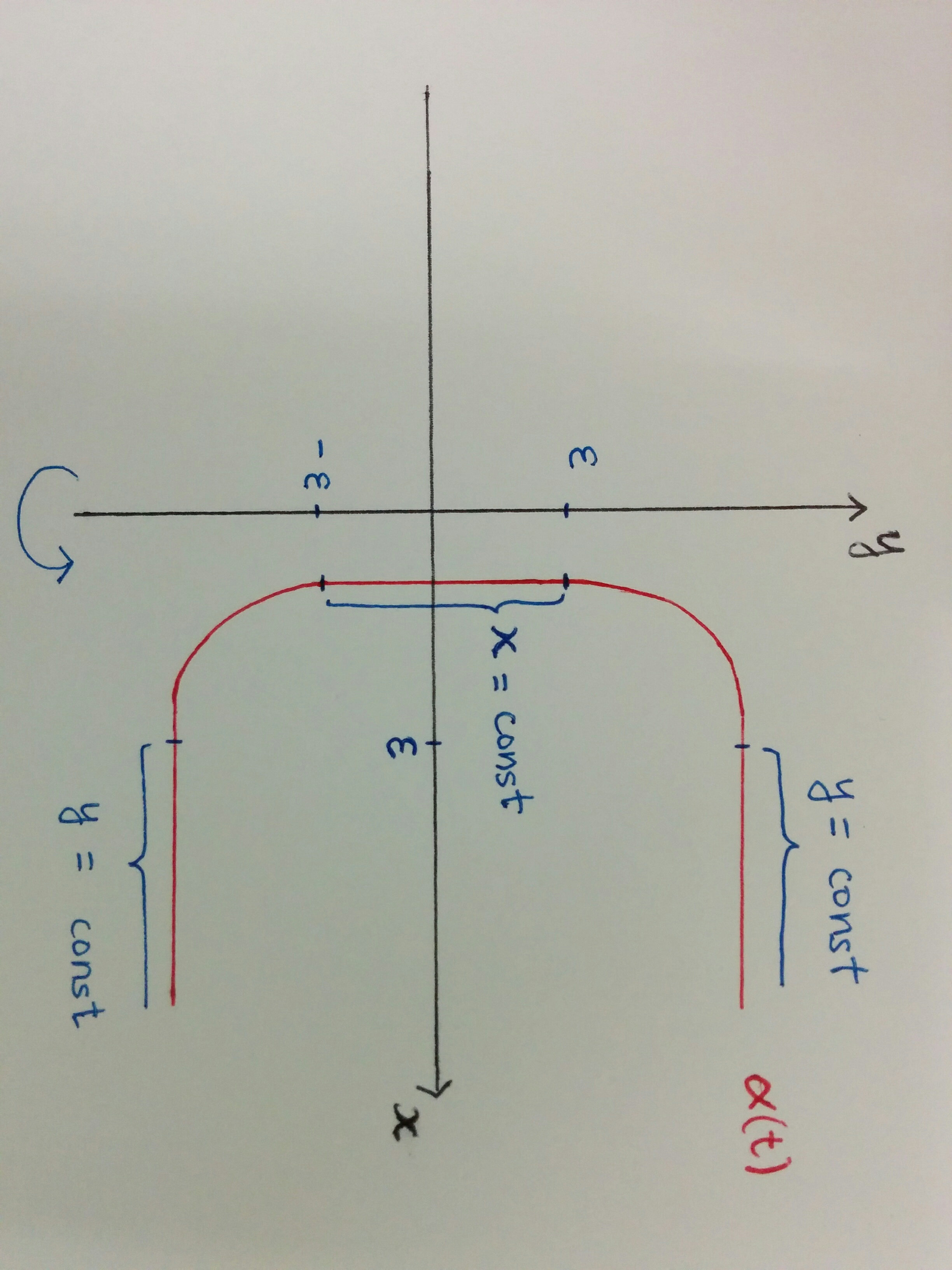}
\caption{}
\end{figure}

We illustrate the simplest, nontrivial case here. The general case will be treated in the body of the paper.
\newline

\noindent
\textbf{An example.} We construct $M$ by taking two manifolds $M_1$ and $M_2$ with isometric totally geodesic submanifolds $T_1$ and $T_2$ (respectively) and gluing the complement of an $\varepsilon$-neighborhood of $T_1$ to the complement of an $\varepsilon$-neighborhood of $T_2$. 
\newline

Let $M_1 = S\times S$ be the product of two copies of a punctured torus. Take a complete hyperbolic metric on $S$ with finite area and let $a$ be a geodesic loop in $S$. Modify the metric smoothly on a regular neighborhood of $a$ in $S$ 
and rescale it if necessary to make it a product $(-1,1)\times \mathbb{S}^1$ without creating positive curvature on $S$. Give $M_1$ the product of this metric. Then $T_1 := a\times a$ is a flat, square 2-torus and has a neighborhood isometric to the product $\mathbb{D}^2\times T_1$. 
\newline

Let $M_2$ be obtained by taking a finite volume, complete, hyperbolic $4$-manifold $H$ with at least three torus-cusps $C_1$, $C_2$ and $C_3$, truncating $C_2$ and $C_3$ and gluing  $\partial C_2$ to $\partial C_3$ via an affine diffeomorphism. Assume for simplicity that the cross sections of each of these cusps are homothetic to the flat, square, 3-torus $\mathbb{T}^3$, so that the gluing can be done via an isometry and gives $M_2$ a bounded nonpositively curved metric. This is standard but we will explain it in the next section. In fact, one can make it so that the metric on $M_2$ is a product $(-1,1)\times \mathbb{T}^3$ on a neighborhood of where the gluing takes place. Now, there is a square $2$-torus $T_2$ factor in $\mathbb{T}^3$, so $T_2$ has a neighborhood isometric to the product $\mathbb{D}^2\times T_2$. 
\newline


Let $M$ be obtained by gluing the complement of the $\varepsilon$-neighborhood of $T_1$ to the complement of the $\varepsilon$-neighborhood of $T_2$ along the boundaries. 
After smoothing out the metric as explained above, we obtain a finite volume, bounded nonpositively curved manifold $M$ with two kinds of cusps, one corresponding to the end of $M_1$, and the other corresponding to the cusp $C_1$ of $M_2$. 
In this example, $\partial\widetilde{M}$ is homotopically equivalent to a graph $\Sigma$ with infinitely many components, each component either contractible or of infinite type (homotopy equivalent to an  infinite wedge of circles). The first kind of cusp looks like a 2-dimensional flat sector from afar and is responsible for the infinite type components in $\Sigma$. The second kind looks like a ray from afar and contributes the contractible components in $\Sigma$.  
\newline

All the simplifying assumptions made above can be taken care of in general when no such assumptions are made. This is dealt with in the rest of the paper and is not difficult. 
\newline

\begin{figure}
\label{funnel}
\centering
\includegraphics[scale=0.42]{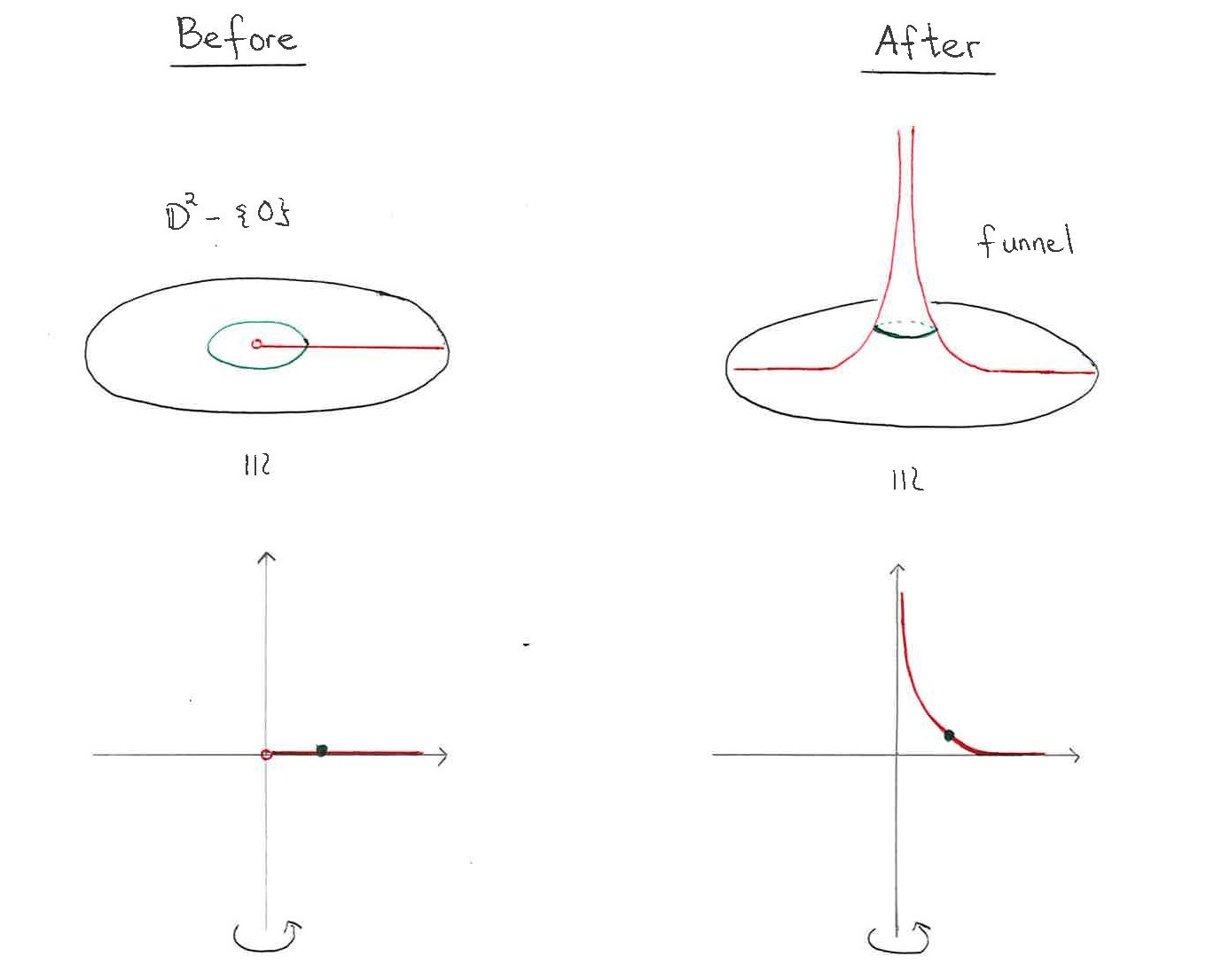}
\caption{}
\end{figure}

A simpler construction that gives a manifold very similar to the manifold $M$ above can be obtained by taking $(M_1-T_1)$ and stretching out the metric in a neighborhood of $T_1$ to make it complete and have finite volume without creating positive curvature. Since the metric on $M_1$ is a product $\mathbb{D}^2\times T_1$, this can be achieved if one can stretch out the metric on $(\mathbb{D}^2-\{0\})$ to obtain a complete, bounded nonpositively curved metric with finite area. This clearly can be done and is illustrated in Figure 2. 
This example is a good example but we did not discuss it above because it does not illustrate every step in the construction given in this paper. But we would like to note that this is a counterexample to a conjecture of Farb on geometric rank 1 manifolds and we will discuss this in Subsection \ref{farbconjecture}.

\begin{acknowledgement}
The second author would like thank Bena Tshishiku for asking her whether $\partial\widetilde{M}$ should be homotopy equivalent to a building.
\end{acknowledgement}

\section{Preliminaries}
As mentioned in the introduction, the construction involves assembling nonpositively curved manifolds containing totally geodesic tori of low codimension. We now describe one way to obtain such manifolds.

\subsection{The building blocks $N_n$ (Hyperbolic straightjackets)}
Start with a complete, finite volume, connected, hyperbolic $n$-manifold $H_n$. After passing to a finite cover, if necessary,  we may assume that $H_n$ has at least three cusps, at least two of which (called $C_+$ and $C_-$) are homeomorphic to $\mathbb T^{n-1}\times(0,\infty)$. Then, the manifold $H_n\setminus(C_+\cup C_-)$ has two boundary components $\partial C_+\cong\mathbb T^{n-1}\cong\partial C_-$. Moreover, the induced metrics on $\partial C_+$ and $\partial C_-$ are flat. Now, let $N_n=(H_n\setminus(C_+\cup C_-))/\partial C_+\sim\partial C_-$ be a manifold obtained by gluing the boundaries $\partial C_+$ and $\partial C_-$ by an affine diffeomorphism.

\begin{proposition}\label{buildingblocks}
For any $r>0$, the manifold $N_n$ has a complete, finite volume, Riemannian metric of bounded non-positive curvature in which a regular neighborhood of $\partial C_+$ is isometric to $\mathbb T^{n-1}\times(-r,r)$, where $\mathbb T^{n-1}$ is a square flat torus.
\end{proposition}  


First, note in the case when $\partial C_+$ and $\partial C_-$ are square, flat tori and the affine diffeomorphism is an isometry,  this is not hard. The hyperbolic metric near $\partial C_+$ or $\partial C_-$ is a warped product and has the form
\[ g_{\text{hyp}} = e^{-2t}g_0 + dt^2,\]
where $g_0$ is a square, flat metric on $\mathbb{T}^{n-1}$ and for some $a$, the slice $t= a$ corresponds to where $\partial C_+$ or $\partial C_-$ is. So around where $\partial C_+$ and $\partial C_-$ are identified,  the metric, after reparametrizing $t$ via a shift by $a$, is 
\[e^{2|t|-2a}g_0 + dt^2\]
on $\mathbb{T}^{n-1}\times[-1,1]$, which is not smooth at $t =0$. But one can replace the warping function $e^{2|t|-2a}$ by a smooth, convex function that, for some small enough $\varepsilon$, agrees with $e^{2|t|-2a}$ outside $(-2\varepsilon,2\varepsilon)$ and that is equal to a positive constant on $(-\varepsilon,\varepsilon)$. Change the range\footnote{In other words, we replace the cylinder $(-\varepsilon,\varepsilon)\times\mathbb{T}^{n-1}$ by the cylinder $(-r,r)\times\mathbb{T}^{n-1}$.} $(-\varepsilon, \varepsilon)$ of $t$-parameter  to $(-r, r)$ but keep the metric otherwise the same to get a desired metric. The fact that the resulting metric has nonpositive curvature is a direct application of the Bishop-O'Neil formula (\cite{bishoponeil}). The relevant point is:
\newline

\textit{If one warps a nonpositively curved metric by a convex function, one gets a nonpositively curved metric.}
\newline

In the general case, the main point is to first interpolate between the square flat metric $g_0$ on $\mathbb T^{n-1}$ and another flat metric $g_1$ on $T^{n-1}$ so that the problem reduces to the above. That is, consider the following metric $g$ on $\mathbb T^{n-1}\times[0,\infty)$.
\[ g = e^{-2t}(h(t)g_0 + (1-h(t))g_1) +dt^2,\]
for some smooth function $h\colon [0,\infty) \rightarrow [0,1]$ such that $h(t) = 0$ when $t$ is close to $0$ and $h(t) = 1$ when $t>l$, for some $l$. One can pick $l$ large enough and an appropriate $h$ so that $g$ has nonpositive curvature as shown in Lemma 2.2 of \cite{aravindafarrell}. Truncate that cusp at $t = a> l$ and apply the above special case to get the desired metric. 


\subsection{A special case}
We will first prove Theorem \ref{maintheorem} in the special case when $n$ is even, $i=0$ and $j=n/2-1$. This is done by inductively constructing manifolds $M_n$ satisfying (\ref{range}) in Proposition \ref{special case}, below. In order to facilitate the induction, the manifolds $M_n$ need to have the additional properties (\ref{ends}), (\ref{flat}), and (\ref{connected}).
\begin{proposition}\label{special case}
For even $n$, there is 
a tame, complete, finite volume, $n$-manifold $M_n$ of bounded non-positive curvature so that
\begin{enumerate}
\item
\label{range}
$\overline H_k(\partial\widetilde M_n)\not=0$ for $k<n/2$, 
\item
\label{ends}
$M_n$ has at least two ends,
\item
\label{flat}
$M_n$ contains an isometrically embedded $T:=\mathbb T^{n-1}\times(-1,1)$, where $\mathbb T^{n-1}=(\mathbb{S}^1)^{n-1}$ is a square flat torus of injectivity radius $1$, and
\item
\label{connected}
$M_n\setminus T$ is connected. 
\end{enumerate} 
\end{proposition}

\subsection{The general case} Theorem \ref{maintheorem} follows from Proposition \ref{special case} by taking products with circles and non-compact surfaces. The key to showing this is the following product formula. 

\subsubsection{\textbf{Product formula}}
If $M$ and $N$ are tame, aspherical manifolds, then one has the following product formula
\begin{equation}
\label{product}
\partial\widetilde{(M\times N)}\sim\partial\widetilde M*
\partial\widetilde N,
\end{equation}
where the symbol $\sim$ denotes homotopy equivalence.
\begin{remark}
This follows from
\begin{equation*}
\begin{array}{ccccc}
\partial\widetilde M*\partial\widetilde N&=&\partial\widetilde M\times Cone(\partial\widetilde N)&\bigcup_{\partial\widetilde M\times\partial\widetilde N}&Cone(\partial\widetilde M)\times\partial\widetilde N,\\
&&&&\\
\partial(\widetilde M\times\widetilde N)&=&\partial\widetilde M\times \widetilde N&\bigcup_{\partial\widetilde M\times\partial\widetilde N}&\widetilde M\times\partial\widetilde N.
\end{array}
\end{equation*}
\end{remark} 

\subsubsection{\textbf{Shifting dimensions via products with circles and surfaces}} Note that for a {\it non-compact} surface $\Sigma$ the cover $\partial\widetilde\Sigma$ is homotopy equivalent to an infinite union of points, which we will write as $\partial\widetilde\Sigma\sim\vee_{i=1}^{\infty}S^0$. Therefore, $\partial(\widetilde{M\times\Sigma})\sim\partial\widetilde M*(\vee_{i=1}^{\infty}S^0)\sim\vee_{i=1}^{\infty}(\partial\widetilde M*S^0)$. So 
\begin{eqnarray}
\label{surface}
\overline H_*(\partial(\widetilde{M\times\Sigma}))&\cong&\bigoplus_{i=1}^{\infty}\overline H_{*-1}(\partial\widetilde M).
\end{eqnarray}
It is also clear that $\partial(\widetilde {M\times S^1})\sim\partial\widetilde M$ so we have
\begin{eqnarray}
\label{circle}
\overline H_*(\partial(\widetilde{M\times S^1}))&\cong&\overline H_*(\partial\widetilde M).
\end{eqnarray}
\subsubsection{\textbf{Proof of Theorem \ref{maintheorem} given Proposition \ref{special case}}}
The proposition gives a $2(j-i+1)$-dimensional manifold $M_{2(j+i-1)}$ whose homology $\overline H_k(\partial\widetilde M_{2(j+1-i)})$ doesn't vanish precisely in the band of dimensions $0\leq k\leq j-i$. Crossing with $i$ non-compact surfaces shifts this band into the desired dimension range $i\leq k\leq j$ (by formula (\ref{surface})) and then crossing with $n-2j-2$ circles raises the dimension of the manifold to $n$ without affecting the band (by formula (\ref{circle})). So, the resulting manifold
\begin{equation}
M=M_{2(j+1-i)}\times(\Sigma)^i\times(S^1)^{n-2j-2},
\end{equation}
satisfies the conclusions of Theorem \ref{maintheorem}.



\section{Proof of Proposition \ref{special case}}
The manifolds $M_n$ are constructed inductively, as follows. 
\subsection{Base case} Topologically, the base case $M_2$ is a twice-punctured torus. Start with a hyperbolic metric on $M_2$. In this metric the two punctures appear as cusps. Let $b$ be a geodesic\footnote{It is not important that $b$ is a geodesic. We could take any path.} that starts in one cusp and ends in the other cusp, and $a$ a non-separating closed geodesic loop that does not intersect $b$. Let length$(a) \geq 2$ and modify the metric so that it is a flat cylinder on an $1$-neighborhood of $a$, hyperbolic outside of a compact set, and still non-positively curved.\footnote{We can do this without changing the length of $a$.} It is easy to see that $M_2$ with this metric satisfies the conditions in the proposition. 
\subsection{Inductive step}
Suppose we have constructed $M_{n-2}$. 
We need to build $M_n$. 
Look at $M_{n-2}\times M_2$. 
It contains an isometrically embedded 
\begin{eqnarray*}
\mathbb T^{n-3}\times(-1,1)\times a\times(-1,1)&\cong&\mathbb T^{n-2}\times (-1,1)^2\\
&\supset&\mathbb T^{n-2}\times \mathbb D^2.
\end{eqnarray*}
On the other hand, suppose that $N_n$ is an $n$-dimensional ``building block" described above, i.e. a manifold obtained from a hyperbolic manifold by gluing a pair of cusps together so that they give an isometrically embedded copy of  
\begin{eqnarray*}
\mathbb T^{n-1}\times(-1,3)&\supset&\mathbb T^{n-2}\times \mathbb S^1\times\left((-1,1)\coprod(1,3)\right)\\
&\supset&\left(\mathbb T^{n-2}\times\mathbb D^2\right)\coprod\left(\mathbb T^{n-1}\times(1,3)\right).
\end{eqnarray*}
The ``$\mathbb T^{n-2}\times\mathbb D^2$'' is used in codimension two surgery, and the ``$\mathbb T^{n-1}\times(1,3)$'' implies that the resulting manifold $M_n$ will have property (\ref{flat}), which let us continue the induction. 
Also recall that $N_n$ has at least one cusp that is not glued to anything. We claim that the manifold 
\begin{equation}
M_n:=\left[N_n\setminus(\mathbb T^{n-2}\times\mathbb D^2)\right]\bigcup_{\mathbb T^{n-2}\times S^1}\left[(M_{n-2}\times M_2)\setminus(\mathbb T^{n-2}\times \mathbb D^2)\right]
\end{equation}
obtained by taking the ``connect sum along $\mathbb T^{n-2}$''
has a complete, finite volume metric of bounded nonpositive curvature. We explain this in the following subsection. 

\subsection{Flat, codimension two surgery in nonpositive curvature}
Suppose $M$ and $N$ are complete, finite volume manifolds of bounded nonpositive curvature and $S\subset M$ a totally geodesic submanifold. Suppose further that a regular neighborhood of $S$ is {\it isometric} to $S\times \mathbb D^2$.

\bigskip
\noindent
\textbf{Cusps.} The manifold $M\setminus S$ has a complete, finite volume, nonpositively curved metric of bounded nonpositive curvature obtained by replacing $S\times (\mathbb D^2 - \{0\})$ by $S\times\funnel$, where a \emph{funnel} is defined as follows. 


\begin{definition}[Funnel]
Let $f \colon (0,1] \rightarrow \R$ be a smooth, strictly convex, non-negative function that satisfies the following properties.
\begin{itemize}
\item[(i)] $f(x) = 0$ when $x\geq 1/2$.
\item[(ii)] $f(x) \rightarrow\infty$ as $x\rightarrow 0$.
\item[(iii)] $\displaystyle{\int_0^1 f(x)dx <\infty}$.
\end{itemize}
Let \underline{$\funnel$} be the surface of revolution obtained by rotating the graph of $f(x)$ around the $y$-axis. Then it is diffeomorphic to $\mathbb D^2-\{0\}$ but has negative Gaussian curvature (because $f(x)$ is strictly convex) and finite area (because of condition (iii) above). See Figure 2. 
\end{definition}

\begin{remark}
We use this in the alternative construction (subsection \ref{variant}) which works in special cases, but do not need it for the proof of Theorem \ref{maintheorem}. 
\end{remark}

\begin{definition}[2-Sided Funnel]
A \underline{2-sided funnel} is the surface of revolution obtained by rotating in the curve $\alpha(t)$, where $\alpha(t)$ is a smooth curve defined as in Figure \ref{graph}, around the $y$-axis. 
A 2-sided funnel is diffeomorphic to $\mathbb{S}^1\times(-1,1)$ but has negative Gaussian curvature. 
\end{definition}

\noindent
\textbf{Codimension two connect sum.} If $N$ also contains an isometrically embedded copy of $S\times \mathbb D^2$, then the {\it $S$-connect sum} 
\begin{equation}
\label{connectsum}
M\#_{S}N:=\left[M\setminus (S\times \mathbb D^2)\right]\cup_{S\times \mathbb S^1}\left[N\setminus(S\times \mathbb D^2)\right]
\end{equation}
has a complete, finite volume metric of bounded nonpositive curvature. After cutting out the regular neighborhoods $S\times \mathbb D^2$ from both manifolds, the metric is obtained by inserting a tube that looks topologically like $S\times(\mathbb S^1\times(0,1))$ but metrically looks like $S\times\{\mbox{two sided funnel}\}$. 
\begin{remark}
In the notation of equation (\ref{connectsum}),   
$$M_n=N_n\#_{\mathbb T^{n-2}}(M_{n-2}\times M_2),$$ so $M_n$ has a complete finite volume metric of bounded nonpositive curvature. 
\end{remark}

\section{Properties of the manifold $M_n$}
The manifold $M_n$ contains the isometrically embedded $T:=\mathbb T^{n-1}\times(1,3)$, which shows property (\ref{flat}). The space $N_n\setminus T$ is connected (because it is homotopy equivalent to the original connected hyperbolic manifold $H_n$ we had before we glued two of its cusps together) and the product $M_{n-2}\times M_2$ is connected (the factors $M_{n-2}$ and $M_2$ are connected because they satisfy property (\ref{connected})) so the space 
\begin{eqnarray*}
M_n\setminus T
&=&[(N_n\setminus T)\setminus(\mathbb T^{n-2}\times \mathbb D^2)]\bigcup_{\mathbb T^{n-2}\times \mathbb S^1}[(M_{n-2}\times M_2)\setminus(\mathbb T^{n-2}\times \mathbb D^2)]
\end{eqnarray*} 
obtained via the codimension two surgery is also connected. This proves property (\ref{connected}).
\newline


Since both $N_n$ and $M_{n-2}\times M_2$ have ends, the manifold $M_n$ has at least two ends. This shows property (\ref{ends}). It also implies that $\partial\widetilde M_n$ has at least two components, so
$$
\overline H_0(\partial\widetilde M_n)\not=0.
$$ 
It remains to establish the positive dimensional cases of property (\ref{range}).
\subsection{Computing $H_{>1}(\partial\widetilde M_n)$} 

Next, let $z$ be a {\it connected} homology cycle representing a nontrivial homology class in $H_{k}(\partial\widetilde{M}_{n-2})$ for $0<k<n/2-1$. Let $\tilde b$ be a lift of the path connecting the two ends of the twice punctured torus, and $b^+$ and $b^-$ its endpoints. Look at the suspended cycle $\Sigma z=z*\{b^+,b^-\}$. Since $z$ is connected, the suspended cycle $\Sigma z$ is simply connected. Therefore, a map $\Sigma z\ra\partial\widetilde M_{n-2}*\partial\widetilde M_{2}\sim\partial(\widetilde{M_{n-2}\times M_2})$ which represents the non-trivial $(k+1)$-homology class $[\Sigma z]\in H_{k+1}(\partial(\widetilde{M_{n-2}\times M_2}))$ lifts to a component of $\partial\widetilde M_n$. So, for $1<k+1<n/2$ we have
$$
H_{k+1}(\partial\widetilde M_n)\not=0.
$$ 

\subsection{Computing $H_1(\partial\widetilde M_n)$} Since $M_{n-2}$ has two ends and $M_{n-2}\setminus T$ is connected, we can find a path $\beta:[0,1]\ra M_{n-2}\setminus T$ connecting two different ends of $M_{n-2}$. Let $z=\partial\tilde\beta\in\overline{H}_0(\partial\widetilde M_{n-2})$ be the non-trivial zero cycle obtained as the boundary of a lift $\tilde\beta$ of $\beta$. Then, the image of $\Sigma z=\{\beta^+,\beta^-\}*\{b^+,b^-\}$ is contractible in $M_n$ because it bounds $\beta\times b$. 
Therefore, in this case the non-trivial homology cycle $[\Sigma z]\in H_1(\partial(\widetilde{M_{n-2}\times M_2}))$ also lifts to a cycle in a component of $\partial\widetilde M_n$, showing that 
$$
H_1(\partial\widetilde M_n)\not=0.
$$ 
In summary, we have shown that $\overline H_k(\partial\widetilde M_n)\not=0$ for $k<n/2$. This proves property (\ref{range}), finishes the proof of Proposition \ref{special case}, and thus also the proof of Theorem \ref{maintheorem}.
  


\section{Miscellaneous}

\subsection{\label{variant}A variant for narrow bands that only uses surfaces}
Note that the regular neighborhood of $a\times a$ inside $M_2\times M_2$ is isometric to $a\times a\times D^2_{\epsilon}$. Replacing $D^2_{\epsilon}$ by a ``funnel'' metric on $D^2_{\epsilon}\setminus\{0\}$, we get a complete, finite volume metric of bounded nonpositive curvature on $$M'_4:=(M_2\times M_2)\setminus(a
\times a).$$ 
The arguments in the previous section apply to show that $\overline H_0(\partial\widetilde M'_4)\not=0$ and $H_1(\partial\widetilde M'_4)\not=0$. Taking products of the manifold $M'_4$ with itself and using the product formula (\ref{product}), we get manifolds $(M'_4)^m$ of dimension $4m$ which have $\overline H_k(\partial\widetilde{(M'_4)^m})\not=0$ precisely when $m-1\leq k\leq 2m-1$. 

\begin{remark}
Taking products with circles $S^1$ and non-compact surfaces $M_2$ we get in this way manifolds $M:=(M'_4)^m\times(M_2)^p\times(S^1)^q$ of dimension $\dim M=4m+2p+q$ for which $\overline H_{k}(\partial\widetilde M)$ is non-zero in a band of dimensions $m-1+p\leq k\leq 2m-1+p$. 
\end{remark}

\subsection{Putting gaps in the band?} A remaining question is whether one can build examples where there are gaps in the set of dimensions in which homology occurs. This is already potentially possible in dimension $n=6$. In this case, the question is whether one can simultaneously have $\overline H_0(\partial\widetilde M)\not=0,H_1(\partial\widetilde M)=0,$ and $H_2(\partial\widetilde M)\not=0$.

\subsection{Large scale geometry}
Denote by $[n]$ the set with $n$ elements. It is easy to see that the main construction gives manifolds that on a large scale look like the Euclidean cone on a complex $C_k$, where $C_k$ is defined inductively via $C_0=[2],C_1=([2]*[2])\coprod[n_4],\dots, C_{k}=(C_{k-1}*[2])\coprod[n_{2k}]$ where $n_{2k}$ is the number of ends of the $2k$ dimensional building block $N_{2k}$.


\subsection{\label{farbconjecture}Geometric rank-1 manifolds with $\pi_1$ generated by a cusp}
Once upon the time, there was a conjecture that said the following.

\begin{conjecture}[Farb]\label{Farbconjecture}
Let $M$ be a tame, complete, finite volume $n$-manifold of bounded nonpositive curvature. Suppose $M$ has geometric rank one. Then there is a loop in $M$ that cannot be homotoped to leave every compact set. 
\end{conjecture}


This is known to be true in dimension $\leq 3$. Below we show that the manifold $(M_1-T_1)$ from the introduction is a $4$-dimensional counterexample to this conjecture, and also build higher dimensional counterexamples.  

\begin{proposition}
There is a counterexample to Conjecture \ref{Farbconjecture} for each $n\geq 4$.
\end{proposition}

We will drop the index ``1" as we longer need it. First, note that the manifold $W := M-T$ has geometric rank $1$ because it is neither a locally symmetric space, nor a product.\footnote{$W$ is not a product of two non-compact manifolds because it has more than one end. It is not a product of a non-compact manifold and a compact manifold, because its two ends do not have a common factor: The end cross sections are $\mathbb T^3$ and the irreducible graph manifold $((\mathbb T^2-D^2)\times S^1)\bigcup_{S^1\times S^1}(S^1\times(\mathbb T^2-D^2))$.} 
Thus, we only need to show that all loops in $W$ can be homotoped to leave all compact sets. This is true because of the following lemma.

\begin{lemma}
Let $(S_1,\partial S_1)$ and $(S_2,\partial S_2)$ be compact, connected manifolds-with-boundary and pick basepoints $s_i\in\partial S_i$. Suppose that $T_i\subset (S_i-\partial S_i)$ are compact non-separating hypersurfaces. Let $S_1\vee S_2 = (S_1\times\{s_2\})\cup(\{s_1\}\times S_2)$. Then the composition
$$
(S_1\vee S_2)\hookrightarrow\partial(S_1\times S_2)\hookrightarrow (S_1\times S_2)-(T_1\times T_2),
$$
is $\pi_1$-onto.
\end{lemma}

\begin{proof}
If $\gamma(t)=(\gamma_1(t),\gamma_2(t))$ is a loop in $(S_1\times S_2)-(T_1\times T_2)$, then the times at which $\gamma_1$ crosses $T_1$ are disjoint from the times at which $\gamma_2$ crosses $T_2$. So, one can decompose $\gamma$ as concatenation $\gamma=\gamma^{(1)}\cdot\dots\cdot\gamma^{(r)}$ of {\it paths} where for each $\gamma^{(k)}$ either the first coordinate path $\gamma^{(k)}_1$ never crosses $T_1$ or the second coordinate path $\gamma^{(k)}_2$ never crosses $T_2$. Using the fact that the $T_i$ are non-separating, we can homotope $\gamma$ to be a concatenation of such {\it loops} (all based at $(s_1,s_2)$). Finally, each such loop $\gamma^{(k)}$ is homotopic to $\gamma^{(k)}_1\cdot\gamma^{(k)}_2$, so we are done. 
\end{proof}

\begin{remark}
Since $T$ has codimension 2 in $M$, there is a loop $\gamma$ in $M$ that goes around $T$. One might wonder how $\gamma$ can be a product of elements in $S_1\vee S_2$. 
Let $b_i$ be a loop in $S_i$ based at $s_i$ that intersects transversely with $T_i$ precisely once. We claim that $\gamma = [b_1,b_2] = b_1b_2b_1^{-1}b_2^{-1}$. To see this, observe that $T': = b_1\times b_2$ is an embedded torus in $M$ that intersects $T$ transversely at exactly one point $p$. So $\gamma$ can be taken to be a loop in $T'$ that goes around $p$. Removing $T$ from $M$ results in removing $p$ from $T'$. Since $T'-\{p\}$ is a punctured torus, the loop $\gamma$, which goes around the puncture, must be the commutator of the generators $b_1$ and $b_2$.
\newline
\end{remark}

\noindent
\textbf{Higher dimensional counterexamples} can be constructed in a very similar manner. In dimension $n\geq 4$, let $S_1$ be the punctured torus as before, and let $T_1=a_1$. Let $S_2$ be the building block $N_{n-2}$ and let $T_2$ be $\mathbb{T}^{n-3}$, the square flat torus in $N_{n-2}$ in Proposition \ref{buildingblocks}. 
The manifold $W:=(S_1\times S_2)-(T_1\times T_2)$ is an $n$-dimensional counterexample to Conjecture \ref{Farbconjecture}. It has geometric rank one for the same reasons as in the above example. To see that $\pi_1(W)$ is generated by loops coming from the end of $M$, we apply the above lemma. 

\subsection{A thick-thin conjecture for nonpositively curved manifolds} We would like to suggest the following replacement for Conjecture \ref{Farbconjecture}.
\begin{conjecture}\label{Avramidiconjecture}
Let $M$ be a tame, complete, finite volume $n$-manifold of bounded\footnote{The conjecture is not true without the lower curvature bound. There is a complete, finite volume, negatively curved metric on the product $\Sigma\times\mathbb R$, where $\Sigma$ is a closed surface with genus $g \geq 2$ (\cite{nguyenphannegative}).} nonpositive curvature. Then there is a compact subset $C\subset M$ that cannot be homotoped to leave every compact set.
\end{conjecture} 
Note that this conjecture makes sense (and is most easily stated) for general finite volume manifolds of bounded nonpositive curvature, not just those of geometric rank one. 
The conjecture is known to be true for locally symmetric manifolds $M$ by a result of Pettet and Souto \cite{pettetsouto}.\footnote{Such locally symmetric manifolds contain {\it maximal periodic flat tori} $\mathbb T^r\ra M$, where $r$ is the $\mathbb R$-rank of the locally symmetric space $M$. Pettet and Souto showed these tori cannot be homotoped into the end (even though loops in such a locally symmetric space can always be homotoped into the end whenever the $\mathbb Q$-rank is $\geq 2$).} Therefore, it is enough to understand it for geometric rank one manifolds. 
\newline

Notice that the examples in this paper are not counterexamples to Conjecture \ref{Avramidiconjecture}. To see this, pick an embedded loop $b_i$ in $S_i$ that intersects $T_i$ transversely exactly once. This is possible since the hypersurfaces $T_i$ are non-separating. Now, let $T_i'$ be a parallel copy of $T_i$. We pick it close to $T_i$, so that $b_i$ intersects $T_i'$ transversely at exactly one point $x_i=b_i\cap T_i'$. Then the closed submanifolds $T'_1\times b_2$ and $b_1\times T'_2$ of $W$ intersect transversely at a single point $x_1\times x_2$. Therefore, the interior of $W$ cannot be homotoped into its end, because if there was such a homotopy $h_t:W\ra W$ with $h_0=\id_W$ and $h_1(W)$ contained in a sufficiently small neighborhood of the end of $W$, then we could move $T'_1\times b_2$ via the homotopy $h_t(T'_1\times b_2)$ to be disjoint from $b_1\times T'_2$. This is a contradiction because intersection number is a homological invariant. Therefore, there is no such homotopy. 

\subsection{A $4$-manifold with three isometric, higher rank ends}
Here is another construction of a complete, finite volume $4$-manifold of bounded nonpositive curvature. This manifold has three ends, and each end looks like the end of a product of punctured tori $(\mathbb T^2\setminus\{0\})\times(\mathbb T^2\setminus\{0\})$.

Inside the flat $4$-torus $\mathbb T^4=\mathbb R^4/\mathbb Z^4$, look at the subspaces 
\begin{eqnarray*}
X^1&=&\{t_1=0,t_2=0\}\cup\{t_3=0,t_4=0\},\\ 
X^2&=&\left\{t_1={1\over 3},t_3={1\over 3}\right\}\cup\left\{t_2={1\over 3},t_4={1\over 3}\right\},\\
X^3&=&\left\{t_1={2\over 3},t_4={2\over 3}\right\}\cup\left\{t_2={2\over 3},t_3={2\over 3}\right\}.
\end{eqnarray*}
Each $X^i$ is isometric to a pair of orthogonally intersecting $2$-tori $\mathbb T^2\cup\mathbb T^2$.
Moreover, $X^1, X^2,$ and $X^3$ are disjoint. In fact, for any $\epsilon<1/6$ the $\epsilon$-neighborhoods of these spaces $X^1_{\epsilon},X^2_{\epsilon},X^3_{\epsilon}$ are isometric and disjoint in $\mathbb T^4$.

Now, fix $\epsilon<1/6$ and let $D^2_{\epsilon}$ be an $\epsilon$-disk centered at $0$ in the $2$-torus $\mathbb T^2$. The punctured $2$-torus $\mathbb T^2\setminus\{0\}$ has a complete, finite volume nonpositively curved metric that is flat on a neighborhood of $\mathbb T^2\setminus D^2_{\epsilon}$ and is a funnel metric on $D^2_{\epsilon}\setminus \{0\}$. Thus, the product $(\mathbb T^2\setminus\{0\})\times(\mathbb T^2\setminus\{0\})=\mathbb T^4\setminus X^1$ has a complete, finite volume nonpositively curved metric that is the original flat metric on a neighborhood of $\mathbb T^4\setminus X_{\epsilon}^1$ and a new metric $g_\dagger$ on $X^1_{\epsilon}\setminus X^1$. Since the $\epsilon$-neighborhoods $X^1_{\epsilon},X^2_{\epsilon},$ and $X^3_{\epsilon}$ are disjoint and isometric, we can copy the $g_\dagger$ metric on $X^1_{\epsilon}\setminus X^1$ to $X^2_{\epsilon}\setminus X^2$ and $X^3_{\epsilon}\setminus X^3$ to get a complete, finite volume, nonpositively curved metric on $$
M:=\mathbb T^4\setminus\{X^1\cup X^2\cup X^3\}.
$$
Finally, since the funnel metric had bounded curvature, this metric on $M$ does, as well.
\newline

This example is interesting because we have cut out the obvious intersections that we used earlier to see the that the interior cannot be homotoped into the end. It is also interesting because it has three identical higher rank ends.

\bibliography{moreends}
\bibliographystyle{amsplain}

\end{document}